\documentclass[journal]{IEEEtran}
\IEEEoverridecommandlockouts                              
\bstctlcite{IEEEexample:BSTcontrol}                                                          
\usepackage{flushend}

\usepackage{romannum}
\usepackage[usenames,dvipsnames]{xcolor}
\usepackage{calrsfs}
\DeclareMathAlphabet{\pazocal}{OMS}{zplm}{m}{n}
\usepackage{graphicx}
\usepackage{graphics} 
\usepackage{epsfig} 
\usepackage{mathptmx}
\usepackage{caption}
\usepackage{subcaption}
\usepackage{times} 
\usepackage{tikz}
\usepackage{tkz-berge}
\usetikzlibrary{fit,shapes}                                     
\usetikzlibrary{positioning}
\usepackage{amsmath, amssymb}

\usepackage{amsthm}
\usepackage{amsfonts} 
\usepackage{setspace}
\usepackage{multirow}
\usepackage{textcomp}
\usepackage[english]{babel}
\usepackage{float} 
\usepackage{algorithmicx}
\usepackage{algorithm}
\usepackage{algpascal}
\usepackage{algc}
\usepackage{algcompatible}
\usepackage{algpseudocode}
\let\oldReturn\Return
\renewcommand{\Return}{\State\oldReturn}
\usepackage{mathtools}
\usepackage{bm}
\usepackage{mathptmx}
\usepackage{times} 
\usepackage{mathtools, cuted}
\usepackage{textcomp}
\usepackage[english]{babel}
\usepackage{rotating}
\usepackage{pgfplots}
\pgfplotsset{compat=1.5}
\usepackage{siunitx}
\usepackage{pgfplotstable}
\usepackage{filecontents}

\newcommand{\B}{\pazocal{B}}
\newcommand{\C}{\pazocal{C}}
\newcommand{\K}{\pazocal{K}}

\newcommand{\D}{\pazocal{D}}
\newcommand{\E}{\pazocal{E}}

\newcommand{\U}{\pazocal{U}}

\renewcommand{\P}{\pazocal{P}}

\renewcommand{\S}{\pazocal{S}}
\newcommand{\x}{{X}}
\renewcommand{\u}{{U}}

\newcommand{\remove}[1]{}

\def \cN{{\mathcal N}}
\def \cB{{\mathcal B}}

\def \cC{{\cal C}}
\def \cS{{\cal S}}
\def \cU{{\cal U}}
\def \cY{{\cal Y}}

\def \*{\star}
\def \10n{\!\!\!\!\!\!\!\!\!\!}

\newcommand{\bK}{\bar{K}}

\newcommand{\R}{\mathbb{R}}
\newcommand{\bA}{\bar{A}}
\newcommand{\bB}{\bar{B}}
\newcommand{\bC}{\bar{C}}

\newtheorem{theorem}{Theorem}
\newtheorem{lem}{Lemma}
\newtheorem{defn}{Definition}

\newtheorem{rem}{Remark}
\newtheorem{prob}{Problem}

\newtheorem{assume}{Assumption}
\newtheorem{prop}{Proposition}

\graphicspath{{Figures/}}

\title{\LARGE \bf Minimum Cost Feedback Selection for Arbitrary Pole Placement in Structured Systems}
\author{Shana~Moothedath,
        Prasanna~Chaporkar
        and~Madhu~N.~Belur
\thanks{The authors are in the Department of Electrical Engineering, Indian Institute of Technology Bombay, India. Email: $\lbrace$shana, chaporkar, belur$\rbrace$@ee.iitb.ac.in.}}

\begin{document}
\maketitle
\thispagestyle{empty}
\pagestyle{empty}

\begin{abstract}
This paper addresses optimal feedback selection for arbitrary pole placement of structured systems when each feedback edge is associated with a cost. Given a structured system and a feedback cost matrix, our aim is to find a feasible feedback matrix of minimum cost that guarantees arbitrary pole placement of the closed-loop structured system. We first give a polynomial time reduction of the weighted set cover problem to an instance of the feedback selection problem and thereby show that the feedback selection problem is NP-hard. Then we prove the inapproximability of the problem by showing that constant factor approximation for the problem does not exist unless the set cover problem can be approximated within a constant factor. Since the problem is hard, we study a subclass of systems whose directed acyclic graph constructed using the strongly connected components of the state digraph is a line graph and the state bipartite graph has a perfect matching. We propose a polynomial time optimal algorithm based on dynamic programming for optimal feedback selection on this class of systems. Further, over the same class of systems we relax the perfect matching assumption, and provide a polynomial time 2-optimal solution based on dynamic programming and a minimum cost perfect matching algorithm.
\end{abstract}
\section{Introduction}\label{sec:intro}
Consider structured matrices $\bA, \bB$ and $\bC$ whose entries are $\*$'s and 0's that represent an equivalence class of control systems whose system dynamics is governed by $\dot{x} = Ax+Bu$, $y=Cx$, where $A \in \R^{n \times n}$, $B \in \R^{n \times m}$ and $C \in \R^{p \times n}$ has the same structure as that of $\bA, \bB$ and $\bC$ respectively. Here, $\R$ denotes the set of real numbers. More precisely, 
\begin{eqnarray}\label{eq:struc}
A_{ij} &=& 0 \mbox{~whenever~} \bA_{ij} = 0,\mbox{~and} \nonumber \\
B_{ij} &=& 0 \mbox{~whenever~} \bB_{ij} = 0,\mbox{~and} \nonumber \\
C_{ij} &=& 0 \mbox{~whenever~} \bC_{ij} = 0.
\end{eqnarray}
Any triple $(A, B, C)$ that satisfies \eqref{eq:struc} is referred as a {\it numerical realization} of the {\it structured system} $(\bA, \bB, \bC)$. Let $P \in \R^{m \times p}$ denote the feedback cost matrix, where $P_{ij}$ is the cost for feeding the $j^{\rm th}$ output to the $i^{\rm th}$ input. It may not be feasible to connect all the outputs to all the inputs. We model this by considering cost of such forbidden connection to be infinite. To denote the feedback connections made, we use matrix $\bK \in \{0, \*\}^{m \times p}$. Now the cost of $\bK$ denoted by $P(\bK)$ is given by $P(\bK) = \sum_{(i,j):\bK_{ij} = \*}P_{ij}$.  For a given $\bK$ we define, $[K]:=\{K: K_{ij} = 0, \mbox{~if~} \bK_{ij} = 0\}$. 

\begin{defn}\label{def:struc}
The structured system $(\bA, \bB, \bC)$ and the feedback matrix $\bK$ is said not to have structurally fixed modes (SFMs) if there exists a numerical realization $(A, B, C)$ of $(\bA, \bB, \bC)$
 such that $\cap_{K \in [K]} \sigma(A + BKC) = \phi$, where $\sigma(T)$ denotes the set of eigenvalues of any square matrix $T$. 
 \end{defn}
In this paper our aim is to find a minimum cost $\bK$ such that the closed-loop system denoted as $(\bA, \bB, \bC, \bK)$ has no SFMs.
Specifically, we wish to solve the following optimization problem: Given $(\bA, \bB, \bC)$ and a feedback cost matrix $P$, let $\K_s := \{\bK \in \{ 0, \*\}^{m \times p}:(\bA,\bB,\bC,\bK) \mbox{ has no SFMs} \}$. Note that $\K_s$ consists of all possible feedback structured matrices $\bK$ such that the closed-loop system $(\bA, \bB, \bC, \bK)$ has no SFMs. 
\begin{prob}\label{prob:one}
Given a structured system $(\bA, \bB, \bC)$ and feedback cost matrix $P$, find
\[ \bK^\* ~\in~ \arg\min_{\10n \bK \in \K_s} P(\bK). \] 
\end{prob}
We refer to Problem~\ref{prob:one} as {\it the minimum cost feedback selection problem for arbitrary pole placement}. Let $p^\* = P(\bK^\*)$ denote the optimal cost of Problem~\ref{prob:one}. For a structured system $(\bA, \bB, \bC)$ with feedback cost matrix $P$, without loss of generality, we assume that $\K_s$ is non-empty. Specifically, $\bK^f \in \K_s$, where $\bK^f_{ij} = \*$ for all $i,j$. Notice that if $\K_s$ is empty, then for every $\bK$ the closed-loop structured system $(\bA, \bB, \bC, \bK)$ has SFMs. However, if $p^\* = +\infty$, then we say that arbitrary pole placement is not possible for that structured system.

It is shown that there does not exist polynomial time algorithm for solving the above problem unless P = NP \cite{CarPeqAguKarPap:15}. In this paper, we prove the hardness of the above problem using the set cover problem. In addition we also show the constant factor inapproximability of the problem. Specifically, we have the following result as one of our main result (see Section~\ref{sec:hard}).

\begin{theorem}\label{th:inapprox}
Consider a structured system $(\bA, \bB, \bC)$ and feedback cost matrix $P$. Then, there does not exist any polynomial time algorithm for solving Problem~\ref{prob:one} that has approximation ratio $b\,{\rm log}\,n$ for $0< b < 1/4$, where $n$ denotes the number of states in the system.
\end{theorem}

Though Problem~\ref{prob:one} is NP-hard and even approximating it within a multiplicative factor of $b\,{\rm log}\,n$ is not feasible on general systems, we give an $O(n^3)$ optimal algorithm based on dynamic programming for a special class of systems. The subclass of systems considered here are those systems in which the directed acyclic graph (DAG) obtained by condensing the strongly connected components (SCCs) of the state digraph to vertices (see Section~\ref{sec:prelim} for more details) forms a line graph\footnote{A line graph is a graph which is a directed elementary path starting at the root vertex and ending at the tip vertex.}. Further, we assume that the state bipartite graph (see Section~\ref{sec:prelim} for more details) has a perfect matching. Note that there exists a wide class of systems called as {\it self-damped} systems \cite{ChaMes:13} that have a perfect matching in $\B(\bA)$, for example consensus dynamics in multi-agent systems and epidemic equations. However, for the class of systems whose state bipartite graph does not have a perfect matching but the DAG of SCCs is a line graph, we give an $O(n^3)$ algorithm based on dynamic programming and minimum cost perfect matching that gives a 2-optimal solution to Problem~\ref{prob:one}. We have the following theorem (see Section~\ref{sec:line} for the proof) as our another main result.
\begin{theorem}\label{th:twostage}
Consider a structured system $(\bA, \bB, \bC)$ and a feedback cost matrix $P$ given as input to Algorithms~\ref{alg:dynamic}~and~\ref{alg:twotage}. Let the DAG of SCCs of the state digraph be a line graph and $p^\*$ denote the optimal cost of Problem~\ref{prob:one}. Then,
\begin{itemize}
\item[(i)] if state bipartite graph has a perfect matching, then output $\bK^a$ of Algorithm~\ref{alg:dynamic} is an optimal solution to Problem~\ref{prob:one}, i.e., $P(\bK^a) = p^\*$.
\item[(ii)] if state bipartite graph does not have a perfect matching, then output $\bK^A$ of Algorithm~\ref{alg:twotage} is a 2-optimal solution to Problem~\ref{prob:one}, i.e., $P(\bK^A) \leqslant 2\,p^\*$.
\end{itemize}
\end{theorem}
The organization of the paper is as follows: in Section~\ref{sec:prelim} we discuss graph theoretic preliminaries used in the sequel, few existing results, related work in this area and our key contributions. In Section~\ref{sec:hard} we prove the hardness of the problem using a reduction of the weighted set cover problem. We also give the negative approximation result of the problem in this section. In Section~\ref{sec:line} we discuss two special classes of linear dynamical systems. For the first class of systems considered we give a polynomial time optimal algorithm based on dynamic programming for solving Problem~\ref{prob:one}. For the second class of systems, we give a polynomial time 2-optimal approximation algorithm for solving Problem~\ref{prob:one}. In Section~\ref{sec:eg} we explain our dynamic programming algorithm given in Section~\ref{sec:line} through an illustrative example. Finally, in Section~\ref{sec:conclu} we give the concluding remarks.
\section{Preliminaries, Existing Results and Related Work}\label{sec:prelim}
In this section we firstly discuss few graph theoretic preliminaries and existing results. Subsequently, we discuss the related work in the area of feedback selection problem and then describe how our work is different. 
\subsection{Preliminaries and Existing Results}
Graph theory is a key tool in the analysis of structured systems since there exist easy-to-check graph theoretic necessary and sufficient conditions for various structural properties of the system. In the case of feedback selection, there exist conditions that can be verified in polynomial time to check if generic pole placement is possible for the resulting system \cite{PapTsi:84}. These conditions are solely based on the closed-loop system digraph denoted as $\D(\bA, \bB, \bC, \bK)$ which is constructed as follows: we define the state digraph $\D(\bA) := \D(V_X, E_X)$ where $V_X = \{x_1,\ldots,x_n\}$ and an edge $(x_j, x_i) \in E_{\x}$ if $\bA_{ij} \neq 0$. Thus a directed edge $(x_j, x_i)$ exists if state $x_j$ can {\it influence} state $x_i$. Now we define the system digraph $\D(\bA, \bB, \bC) := \D(V_{\x}\cup V_{\u}\cup V_{Y}, E_{\x}\cup E_{\u}\cup E_{Y})$, where  $V_{U} = \{u_1, \ldots, u_m \}$ and $V_{Y} = \{y_1, \ldots, y_p \}$. An edge $(u_j, x_i) \in E_{U}$ if $\bB_{ij} \neq 0$ and an edge $(x_j, y_i) \in E_{Y}$ if $\bC_{ij} \neq 0$. Thus a directed edge $(u_j, x_i)$ exists if input $u_j$ can {\it actuate} state $x_i$ and a directed edge $(x_j, y_i)$ exists if output $y_i$ can {\it sense} state $x_j$. Then the closed-loop system digraph $\D(\bA, \bB, \bC, \bK) := \D(V_{\x}\cup V_{\u}\cup V_{Y}, E_{\x}\cup E_{\u}\cup E_{Y}\cup E_{K})$, where $(y_j, u_i) \in E_K$ if $ \bK_{ij}\neq 0$. Here a directed edge $(y_j, u_i)$ exists if output $y_j$ can be {\it fed to} input $u_i$.

A digraph is said to be strongly connected if for each ordered pair of vertices $(v_1,v_k)$
there exists an elementary path from $v_1$ to $v_k$. A strongly connected component (SCC) is a subgraph that consists of a maximal set of strongly connected vertices. Using the SCCs of $\D(\bA)$ we construct a directed acyclic graph (DAG), where each node in the DAG is an SCC of $\D(\bA)$. Also, the edges in the DAG are such that, there exists an edge between two nodes in the DAG if and only if there exists an edge in $\D(\bA)$ that connects two states in those SCCs. For the state digraph $\D(\bA)$, we have the following characterization for SCCs.
\begin{defn}\label{def:scc}
An SCC is said to be \underline{linked} if it has at least one incoming or outgoing edge from another SCC. Further, an SCC is said to be \underline{non-top} \underline{linked} (\underline{non-bottom} \underline{linked}, resp.) if it has no incoming (outgoing, resp.) edges to (from, resp.) its vertices from (to, resp.) the vertices of another SCC.
\end{defn}
Now, using the closed-loop system digraph $\D(\bA, \bB, \bC, \bK)$ the following result has been shown \cite{PicSezSil:84}.
\begin{prop} [\cite{PicSezSil:84}, Theorem 4]\label{prop:SFM} 
A structured system $(\bA, \bB, \bC)$ have no structurally fixed modes with respect to an information pattern $\bK$ if and only if the following conditions hold:

\noindent a)~in the digraph $\D(\bA, \bB, \bC, \bK)$, each state node $x_i$ is contained in an SCC which includes an edge from $E_K$, and 

\noindent b)~there exists a finite node disjoint union of cycles $\C_g = (V_g, E_g)$ in $\D(\bA, \bB, \bC, \bK)$ where $g$ 
belongs to the set of natural numbers such that $V_X \subseteq \cup_{g}V_g$.
\end{prop}
 
Condition~a) can be checked in $O(n^2)$ computations \cite{AhoHop:74} and condition~b) can be checked in $O(n^{2.5})$ computations \cite{PapTsi:84}. If $\D(\bA)$ is a single SCC, then the graph is said to be {\it irreducible}. In such a case satisfying condition~a) in Proposition~\ref{prop:SFM} is trivial as any single $(y_i, u_j)$ edge is enough to satisfy the required. Then, solving Problem~\ref{prob:one} simplifies to satisfying condition~b) optimally which is polynomial \cite{PeqKarPap:15}. For checking condition~b) in Proposition~\ref{prop:SFM} an easy to check condition based on the concept of information paths is given in \cite{PapTsi:84}. However, using the bipartite graph $\B(\bA, \bB, \bC, \bK)$ constructed using the adjacency matrix given in \cite{PicSezSil:84}, there exists a matching condition for checking condition~b), where $\B(\bA, \bB, \bC, \bK) := \B(V_{X'} \cup V_{U'} \cup V_{Y'}, V_{X}\cup V_{U} \cup V_{Y}, \E_{\x} 
 \cup \E_{\u} \cup \E_Y \cup \E_K \cup \E_\mathbb{U} \cup \E_\mathbb{Y})$, where $V_{X'}=\{x'_1, \ldots, x'_n \}$, $V_{U'}=\{u'_1, \ldots, u'_m \}$, $V_{Y'} = \{y'_1, \ldots, y'_p \}$ and $V_{X}=\{x_1, \ldots, x_n \}$, $V_{U}=\{u_1, \ldots, u_m \}$ and $V_{Y} = \{y_1, \ldots, y_p \}$. Also, $(x'_i, x_{j}) \in \E_{\x} \Leftrightarrow (x_j, x_i) \in E_{\x}$, $(x'_i, u_{j}) \in \E_{\u} \Leftrightarrow (u_j, x_i) \in E_{\u}$, $(y'_{j}, x_i) \in \E_{Y} \Leftrightarrow (x_i, y_j) \in E_{Y}$ and $(u'_i, y_{j}) \in \E_{K} \Leftrightarrow (y_j, u_i) \in E_{K}$. Moreover, $\E_\mathbb{U}$ include edges $(u'_i, u_i)$ for $i =1,\ldots,m$ and  $\E_\mathbb{Y}$ include edges $(y'_j, y_j)$ for $j =1,\ldots,p$. Given a bipartite graph $G(V, \widetilde{V}, \E)$, where $V \cap \widetilde{V} = \phi$ and $\E \subseteq V \times \widetilde{V}$, a matching $M$ is a collection of edges $M \subseteq \E$ such that no two edges in $M$ share a common end point. For $|V|= |\widetilde{V}|$, if $|M| =|V|$, then $M$ is said to be a perfect matching, where $|D|$ denotes the cardinality of the set $D$. Using $\B(\bA, \bB, \bC, \bK)$ the following result has been shown \cite{MooChaBel:17}.
 
\begin{prop}[\cite{MooChaBel:17}, Theorem 3]\label{prop:match}
Consider a closed-loop structured system $(\bA, \bB, \bC, \bK)$. Then, the bipartite graph $\B(\bA, \bB, \bC, \bK)$ has a perfect matching if and only if all state nodes are spanned by disjoint union of cycles in $\D(\bA, \bB, \bC, \bK)$.
\end{prop}
In a special case, if the state bipartite graph $\B(\bA) := \B(V_{X'}, V_{X}, \E_{\x})$ has a perfect matching, then all state nodes lie in node disjoint cycles that consists of only $x_i$'s. Thus condition~b) is satisfied even without using any feedback edge. 

Summarizing, in general given a closed-loop structured system $(\bA, \bB, \bC, \bK)$ we can check if the system has SFMs or not in $O(n^{2.5})$ computations. Unlike checking for SFMs, designing minimum cost feedback matrix such that the closed-loop system does not have SFMs is computationally hard. Now we discuss some related work in this area.
\subsection{Related Work}
Structural analysis of systems has achieved research interest these days because of its wide range of applications in the real world systems and complex networks. The strength of these analyses are that they require only the graph of the system to study an equivalence class of systems whose graph pattern are the same, i.e., many structural properties are generic in nature \cite{Mur:87}. The system properties like controllability and observability of structured systems is introduced in \cite{Lin:74}. Here, we discuss only the most relevant literature. Structural analysis for various other problems can be found in \cite{LiuBar:16} and references therein.  

The concept of fixed modes under static feedback
structural constraints is introduced in \cite{WanDav:73}. Later, algebraic characterization of the  fixed modes is given in \cite{AndCle:81}. In the case of structured systems, non-existence of structurally fixed modes is a generic property. Hence, if a structured system has no SFMs, then arbitrary pole placement is possible for {\it almost all} numerical realizations of it \cite{SezSil:81}. Thus structural analysis of systems for no SFMs criteria is helpful in studying an equivalence class of control systems. There are necessary and sufficient graph theoretic conditions for checking the existence of SFMs in structured systems \cite{PicSezSil:84}, \cite{PapTsi:84}. Here, optimal selection of a feedback matrix for arbitrary pole placement is our main focus. Next we describe existing work in this area.

Feedback selection for arbitrary pole placement is considered in \cite{Sez:83}, \cite{UnySez:89}, \cite{PeqKarPap:15}, \cite{PeqKarAgu_2:16}, \cite{CarPeqAguKarPap:15} and \cite{MooChaBel:17}. Reference \cite{Sez:83} considers sparsest feedback selection for a given structured system $(\bA, \bB, \bC)$. The authors proposed a method for finding the minimum set of feedback edges by determining the minimum number of inputs and outputs, which itself is an NP-hard problem to solve \cite{PeqSouPed:15}. Reference \cite{UnySez:89} considers optimal feedback selection when each feedback edge is associated with a cost and proposes an algorithm that gives a sub-optimal solution. However, the approach given is by solving a NP-hard problem, the multi-commodity network flow problem. In short, the scheme given in \cite{UnySez:89} is sub-optimal and not polynomial. Due to this, the algorithm proposed in \cite{Sez:83} or  \cite{UnySez:89} is not a polynomial time solution to the feedback selection problem. Given a structured state matrix $\bA$, finding jointly sparsest $\bB$, $\bC$ and $\bK$ such that the closed-loop system has no SFMs is considered in \cite{PeqKarAgu_2:16}. Finding a minimum cost input-output set and  feedback matrix for a given structured system $(\bA, \bB, \bC)$ such that the resulting closed-loop system has no SFMs is considered in \cite{PeqKarPap:15}, when input, output and every feedback edge  are associated with costs. However, because of the NP-hardness of the problem, a special class of systems where the state digraph is irreducible is considered by the authors. Given $(\bA, \bB, \bC)$ and costs corresponding to each input and output, finding a minimum cost input-output set such that by connecting all outputs in the set to all inputs in the set, the resulting closed-loop system has no SFMs is considered in \cite{MooChaBel:17}. Since the problem is NP-hard, the authors of \cite{MooChaBel:17} proposed an order optimal polynomial time approximation algorithm. This paper deals with finding a minimum cost $\bK$ for arbitrary pole placement for the given structured system $(\bA, \bB, \bC)$ when each feedback edge has a cost associated with it. If all costs are non-zero and uniform, then the problem boils down to finding a sparsest $\bK$. This problem is considered in \cite{CarPeqAguKarPap:15}. The authors of \cite{CarPeqAguKarPap:15} showed the NP-hardness of the problem using reduction of a known NP-complete problem, the input-output decomposition problem. However, there are no known approximation algorithms available for the input-output decomposition problem, nor there are any inapproximability results in the literature, to the best of our knowledge. 
  
\noindent Next we list our key contributions.

\noindent $\bullet$ We prove that even when $\B(\bA)$ has a perfect matching and $\D(\bA)$ has only one non-top linked SCC, Problem~\ref{prob:one} is NP-hard.

\noindent $\bullet$ We prove that Problem~\ref{prob:one} cannot be approximated within a multiplicative factor $b\,{\rm log\,}n$, where $n$ denotes the number of states in the system and $0 < b < \frac{1}{4}$. 

\noindent $\bullet$ We give a polynomial time optimal algorithm of complexity $O(n^3)$ for solving Problem~\ref{prob:one} for a special class of systems whose DAG of SCCs is a line graph and $\B(\bA)$ has a perfect matching. 

\noindent $\bullet$  We give a polynomial time 2-optimal approximation algorithm of complexity $O(n^3)$ for solving Problem~\ref{prob:one} for systems whose DAG of SCCs is a line graph and $\B(\bA)$ does not have a perfect matching.

Using the details given in this section, now we give a hardness result of Problem~\ref{prob:one} in the next section.
\section{Hardness Results}\label{sec:hard}
Given a structured system $(\bA, \bB, \bC)$ finding a sparsest feedback matrix that satisfies arbitrary pole placement is known to be NP-hard \cite{CarPeqAguKarPap:15}. The authors proved the NP-hardness using a reduction of the input-output decomposition problem \cite{Tar:84} to an instance of Problem~\ref{prob:one} where $\B(\bA)$ has a perfect matching, $\D(\bA, \bB, \bC, \bK^\*)$ has two SCCs and input-output set is dedicated\footnote{Every input (output, resp.) can actuate (sense, resp.) a single state only.}. In this section we prove the hardness of Problem~\ref{prob:one} using another well known NP-hard problem, the set cover problem. The reason to do this is to prove inapproximability result in addition to the NP-hardness of the problem. We prove the inapproximability result by showing that the problem cannot be approximated within a multiplicative factor $b\,{\rm log\,}n$, where $n$ denotes the number of states in the system and $b$ is a constant. The NP-hard result is obtained by reducing the weighted set cover problem to an instance of Problem~\ref{prob:one}. We prove that the problem is NP-hard even for the case where the state bipartite graph $\B(\bA)$ has a perfect matching and there is only one non-top linked SCC in $\D(\bA)$. We first detail the weighted set cover problem for the sake of completeness. Given universe $\U = \{ 1,2, \cdots, N\}$ of $N$ items, $r$ sets $\P = \{\S_1, \S_2, \cdots, \S_r \}$ with $\S_i \subset \U$ and $\bigcup_{i = 1}^r \S_i = \U$ and a weight function $w: \P \rightarrow \R$, the weighted set cover problem consists of finding a set $\S^\* \subseteq \P$ such that $\cup_{\S_i \in \S^\*} \S_i = \U$ and $\sum_{\S_i \in \S^\*}w(i) \leqslant \sum_{\S_i \in \widetilde{\S}} w(i)$ for any $\widetilde{\S}$ that satisfies $\cup_{\S_i \in \widetilde{\S}} = \U$. In order to prove the hardness of the problem we give a reduction of a general instance of the weighted set cover problem to an instance of Problem~\ref{prob:one}.

\begin{algorithm}[t]
  \caption{Pseudo-code for reducing the weighted set cover problem to an instance of Problem~\ref{prob:one}
  \label{alg:set}}
  \begin{algorithmic}
\State \textit {\bf Input:} Weighted set cover problem with universe $\U= \{1,\ldots, N\}$, sets $\P=\{\S_1,\ldots, \S_r\}$ and weight function $w$ 
\State \textit{\bf Output:} Structured system $(\bA, \bB, \bC)$ and feedback cost matrix $P$ 
\end{algorithmic}
  \begin{algorithmic}[1]
  \State Define a structured system $(\bA, \bB, \bC)$ as follows:
  \State  $\bA_{ij} \leftarrow \begin{cases}
\*, \mbox{~for~} i = j,\\
\*, \mbox{~for~} i \in \{1,\ldots,N\}\mbox{~and~} j = N+1, \\
0, \mbox{~otherwise}.
  \end{cases} $\label{step:A}
   \State  $\bB_{i1} \leftarrow \begin{cases}
\*, \mbox{~for~} i = N+1,  \\
0, \mbox{~otherwise}.\label{step:B}
  \end{cases} $ 
     \State  $\bC_{ij} \leftarrow \begin{cases} 
\*, \mbox{~for~} j \in \S_i, \\
0, \mbox{~otherwise}.\label{step:C}
  \end{cases} $ 
\State Define feedback cost matrix $P$ as: 
\State $P_{1j} \leftarrow 
w(j), \mbox{~for~} j \in \{1,\ldots,r\}$.\label{step:cost}
 \State Given a solution $\bK$ to Problem~\ref{prob:one} on $(\bA, \bB, \bC)$, define: 
\State Sets selected under $\bK$, $\S(\bK) \leftarrow \{\S_j:  \bK_{1j} \neq 0\}$, \label{step:set}
\State Weight of the set, $w(\S(\bK)) \leftarrow \sum_{\S_i \in \S(\bK)}w(i)$ \label{step:setcost}.
\end{algorithmic}
\end{algorithm}

The pseudo-code showing a polynomial time reduction of the weighted set cover problem to an instance of Problem~\ref{prob:one} is presented in Algorithm~\ref{alg:set}. Consider a general instance of the weighted set cover problem consisting of universe $\U$ with $|\U| = N$, sets $\P = \{\S_1,\ldots,\S_r\}$ and weight $w$. We construct a structured system $(\bA, \bB, \bC)$ that has states $x_1,\ldots, x_{N+1}$, input $u_1$ and outputs $y_1, \ldots, y_r$. $\bA \in \{ 0, \*\}^{(N+1) \times (N+1)}$ is constructed as shown in Step~\ref{step:A}. Notice that in $\D(\bA)$ every state has an edge to itself. In addition, state $x_{N+1}$ has an edge to all other states. Thus $\B(\bA)$ has a perfect matching, $M=\{(x'_i, x_i) \mbox{~for~} i\in \{1,\ldots,N+1 \}\}$. Hence, condition~b) in Proposition~\ref{prop:SFM} is satisfied. Next, $\bB \in \{0, \*\}^{(N+1) \times 1}$ is constructed as shown in Step~\ref{step:B}. We consider a single input $u_1$ that connects to state $x_{N+1}$ only. Now, $\bC \in \{0, \*\}^{r \times (N+1)}$ is constructed as shown in Step~\ref{step:C}. Notice that construction of $\bC$ relies on $\P$. Specifically, $\bC_{ij} = \*$ if $x_j \in \S_i$. Finally, the cost matrix $P$, which gives the costs for feeding outputs $y_j$'s to input $u_1$, is defined as shown in Step~\ref{step:cost}. Note that $\bK = \{\bK_{ij} = \*, \mbox{ for all } i,j\} \in \K_s$. Thus $\K_s$ is non-empty. 

An illustrative example showing the above construction is given in Figure~\ref{fig:sys1}. We solve Problem~\ref{prob:one} on the structured system $(\bA, \bB, \bC)$ and cost matrix $P$  obtained in Algorithm~\ref{alg:set}. Let $\bK$ be a solution. Then the sets selected under $\bK$ and its cost is defined as in Steps~\ref{step:set}~and~\ref{step:setcost} respectively. 
\begin{figure}
\begin{center}
\begin{tikzpicture}[->,>=stealth',shorten >=1pt,auto,node distance=1.65cm, main node/.style={circle,draw,font=\scriptsize\bfseries}]
\definecolor{myblue}{RGB}{80,80,160}
\definecolor{almond}{rgb}{0.94, 0.87, 0.8}
\definecolor{bubblegum}{rgb}{0.99, 0.76, 0.8}
\definecolor{columbiablue}{rgb}{0.61, 0.87, 1.0}

  \fill[almond] (-1,-2) circle (7.0 pt);
  \fill[almond] (-2,-2) circle (7.0 pt);
  \fill[almond] (0,0) circle (7.0 pt);
  \fill[almond] (0,-2) circle (7.0 pt);
  \fill[almond] (1,-2) circle (7.0 pt);
  \fill[almond] (2,-2) circle (7.0 pt);
  \node at (-2,-2) {\small $x_1$};
  \node at (-1,-2) {\small $x_2$};
  \node at (0,0) {\small $x_6$};
  \node at (0,-2) {\small $x_3$};
  \node at (1,-2) {\small $x_4$};
  \node at (2,-2) {\small $x_5$};

  \fill[bubblegum] (0,1) circle (7.0 pt);
   \node at (0,1) {\small $u_1$};
   
  \fill[columbiablue] (-1.5,-3) circle (7.0 pt);
  \fill[columbiablue] (-0.5,-3) circle (7.0 pt);
  \fill[columbiablue] (1.5,-3) circle (7.0 pt);
   
   \node at (-1.5,-3.0) {\small $y_1$};
   \node at (-0.5,-3.0) {\small $y_2$};
   \node at (1.5,-3.0) {\small $y_3$};

  \draw (0,0.75)  ->   (0,0.25);
  \draw (0,-0.25)  ->   (-2,-1.75);
  \draw (0,-0.25)  ->   (-1,-1.75);
  \draw (0,-0.25)  ->   (0,-1.75);
  \draw (0,-0.25)  ->   (1,-1.75);
  \draw (0,-0.25)  ->   (2,-1.75);

  \draw (-2,-2.25)  ->   (-1.5,-2.75);
  \draw (-1,-2.25)  ->   (-1.5,-2.75);
  \draw (-1,-2.25)  ->   (-0.5,-2.75);
  \draw (0,-2.25)  ->   (-0.5,-2.75);
  \draw (1,-2.25)  ->   (1.5,-2.75);
  \draw (2,-2.25)  ->   (1.5,-2.75);
  \draw (0,-2.25)  ->   (1.5,-2.75);  
\path[every node/.style={font=\sffamily\small}]
(-2.05,-1.75) edge[loop above] (-2,-1)
(-1.05,-1.75)edge[loop above] (-1,-1)
(0.1,0.25) edge[loop above]  (0,0)
(0.1,-1.75) edge[loop above]  (0,-1)
(1.0,-1.75) edge[loop above]  (1,-1)
(2.0,-1.75) edge[loop above]  (2,-1);

\end{tikzpicture}
\caption{Digraph $\D(\bA, \bB, \bC)$ constructed using Algorithm~\ref{alg:set} for a weighted set cover problem with $\U = \{1,\ldots,5\}$, $\P= \{\S_1, \S_2, \S_3\}$, where $\S_1 = \{1,2 \}$, $\S_2 = \{2,3\}$ and $\S_3 = \{3,4,5\}$.}
\label{fig:sys1}
\end{center}
\end{figure}
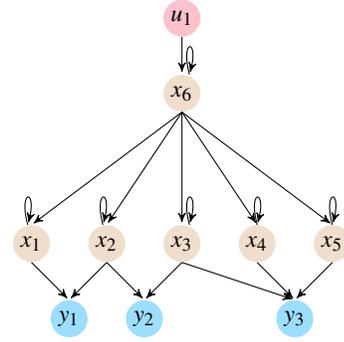 
Now we formulate and prove the following.

\begin{lem}\label{lem:set_prob}
Consider the weighted set cover problem $\U, \P, w$ and a structured system $(\bA, \bB, \bC)$ and feedback cost matrix $P$ constructed using Algorithm~\ref{alg:set}. For this structured system, $\bK \in \K_s$ if and only if $\S(\bK)$ covers $\U$. Moreover, $w(\S(\bK)) = P(\bK)$.
\end{lem}
\begin{proof}
\noindent{\bf Only-if part:} We assume $\bK \in \K_s$ and then show that $\S(\bK)$ is a cover. Given $\bK$ is a solution to Problem~\ref{prob:one}. Note that the structured system $(\bA, \bB, \bC)$ constructed in Algorithm~\ref{alg:set} has a perfect matching in $\B(\bA)$. Thus condition~b) in Proposition~\ref{prop:SFM} is satisfied without using any feedback edge. However, $\bK$ satisfies condition~a). Now we need to prove $\cup_{\S_i \in \S(\bK)} = \U = \{1, \ldots, N \}$. Suppose not. Then there exists an element $j \in \U$ that is not covered by $\S(\bK)$. Let $\S(\bK)$ consist of sets $\S_{i_1}, \ldots, \S_{i_k}$ and the corresponding outputs are $y_{i_1}, \ldots, y_{i_k}$. Thus $\bK_{1r} = \*$, for $r \in \{i_1, \ldots, i_k \}$. Since element $j$ is not covered by $\S({\bK})$, there does not exist $y \in \{ y_{i_1}, \ldots, y_{i_k} \}$ that has edge $(x_j, y)$. Thus $x_j$ does not satisfy condition~a) in Proposition~\ref{prop:SFM}. This contradicts the assumption that $\bK$ is a solution to Problem~\ref{prob:one}. This proves the only-if part of the proof.

\noindent{\bf If part:} We assume that $\S(\bK)$ is a cover and then show that $\bK \in \K_s$. Suppose not. Since $\B(\bA)$ has a perfect matching all state nodes lie in disjoint cycles which consists of only state nodes. Thus condition~b) in Proposition~\ref{prop:SFM} is satisfied without using any feedback edge. Thus $\bK \notin \K_s$ implies that there exists a state $x_j$ that does not satisfy condition~a). Let $\bK \in \{0, \*\}^{1 \times (N+1)}$ has $\*$'s at indices $i_1, \ldots, i_k$. That means outputs $y_{i_1}, \ldots, y_{i_k}$ are fed back to input $u_1$. The corresponding sets are $\S_{i_1}, \ldots, \S_{i_k}$. Assume $j \leqslant N$. Since $x_j$ does not satisfy condition~a), there does not exists $r \in \{ i_1, \ldots, i_k\}$ such that edge $(x_j, y_r)$ is present. Then, there exists no set $\S_r \in \{\S_{i_1}, \ldots, \S_{i_k}\}$ such that element $j$ is covered. This contradicts the assumption that $\S(\bK)$ is a cover. Now if $j = N+1$, then $\{i_1, \ldots, i_k\} = \phi$. Thus no output is fed back to $u_1$ and thus $\S(\bK) = \phi$. This contradicts the assumption that $\S(\bK)$ is a cover. Thus $\bK \in \K_s$. This completes the if part of the proof.

Now Step~\ref{step:cost} in Algorithm~\ref{alg:set} gives $w(\S(\bK)) = P(\bK)$ and this completes the proof.
\end{proof}

Next we prove the hardness of Problem~\ref{prob:one} using a reduction of the weighted set cover problem. We show that any instance of the weighted set cover problem can be reduced to an instance of Problem~\ref{prob:one} such that an optimal solution to Problem~\ref{prob:one} gives an optimal solution to the weighted set cover problem.

\begin{theorem}\label{th:NP1}
Consider a structured system $(\bA, \bB, \bC)$ and feedback cost matrix $P$ constructed using Algorithm~\ref{alg:set} corresponding to a weighted set cover problem. Let $\bK^\*$ be an optimal solution to Problem~\ref{prob:one} and $\S(\bK^\*)$ be the cover corresponding to $\bK^\*$. Then, $\S(\bK^\*)$ is an optimal solution to the weighted set cover problem. Moreover, Problem~\ref{prob:one} is NP-hard. 
\end{theorem}
\begin{proof}
Given a general instance of the weighted set cover problem, we first construct a structured system $(\bA, \bB, \bC)$ and feedback cost matrix $P$ using Algorithm~\ref{alg:set}. Now we prove that a feasible solution to Problem~\ref{prob:one} gives a feasible solution to the weighted set cover problem. Then we prove that an optimal solution to Problem~\ref{prob:one} gives an optimal solution to the weighted set cover problem.

 Let $\bK$ be a feasible solution to Problem~\ref{prob:one}. Using Lemma~\ref{lem:set_prob} the sets selected under $\bK$, $\S(\bK)$ covers $\U$. Hence, $\S(\bK)$ is a feasible solution to the weighted set cover problem. For proving optimality, we use a contradiction argument. Let $\bK^\*$ be an optimal solution to Problem~\ref{prob:one}. From Lemma~\ref{lem:set_prob}, $\S(\bK^\*)$ covers $\U$ and $P(\bK^\*) = w(\S(\bK^\*))$. Thus $\S(\bK^\*)$ is a feasible solution to the weighted set cover problem. Now to prove optimality, we show that $w(\S(\bK^\*)) \leqslant w(\S)$ for any $\S$ that satisfies $\cup_{\S_i \in \S}\S_i = \U$. In other words, an optimal solution to Problem~\ref{prob:one} gives an optimal solution to the weighted set cover problem. Suppose not. Then there exists a cover $\widetilde{\S}$ such that $w(\widetilde{\S}) < w(\S(\bK^\*))$. Let $\widetilde{\S}$ consists of sets $\{\S_{i_1}, \S_{i_2}, \ldots, \S_{i_k}\}$ and the corresponding outputs are $\{y_{i_1}, y_{i_2},\ldots,y_{i_k} \}$. Note that there is only one input $u_1$. Connecting $\{y_{i_1}, y_{i_2},\ldots,y_{i_k} \}$ to $u_1$ satisfies condition~a) in Proposition~\ref{prop:SFM}. This is because for any non-bottom linked SCC, say $\cB_k = x_k$ there is some $y \in \{y_{i_1}, y_{i_2},\ldots,y_{i_k} \}$ connecting $x_k$. So, $u_1 \rightarrow x_{N+1} \rightarrow x_k \rightarrow y \rightarrow u_1$ is a cycle and hence $x_{N+1}$ and $x_k$ belong to the same SCC that has $y \rightarrow u_1$ edge. Since $\cB_k$ is arbitrary, condition~a) holds. So, $\{y_{i_1}, y_{i_2},\ldots,y_{i_k} \}$ is a feasible solution which has a cost given by the cost of the set cover. Thus for $\widetilde{\bK}= \{\widetilde{\bK}_{1j} = \*: j \in \{i_1,\ldots,i_k\} \}$, $P(\widetilde{\bK}) < P(\bK^\*)$. This contradicts the assumption that $\bK^\*$ is an optimal solution to Problem~\ref{prob:one}. This proves that an optimal solution to Problem~\ref{prob:one} gives an optimal solution to the weighted set cover problem.
 
Using Lemma~\ref{lem:set_prob} and since any optimal solution to Problem~\ref{prob:one} gives an optimal solution to the weighted set cover problem, Problem~\ref{prob:one} is NP-hard.
\end{proof}

Note that, for the structured system $(\bA, \bB, \bC)$ constructed in Algorithm~\ref{alg:set}, the bipartite graph $\B(\bA)$ has a perfect matching. Thus all state nodes lie in a cycle that consists of only $x_i$'s and thus condition~b) in Proposition~\ref{prop:SFM} is satisfied. Thus Theorem~\ref{th:NP1} implies that satisfying condition~a) optimally itself is NP-hard even for systems that has a single non-top linked SCC. 

Now we give the following lemma to show that an approximate solution to Problem~\ref{prob:one} on the structured system constructed using Algorithm~\ref{alg:set} gives an approximate solution to the weighted set cover problem. 

\begin{lem}\label{lem:epsilon}
Consider the weighted set cover problem and the structured system $(\bA, \bB, \bC)$ and cost matrix $P$ constructed in Algorithm~\ref{alg:set}. For $\epsilon > 1$, if there exists an $\epsilon$-optimal solution to Problem~\ref{prob:one}, then there exists an $\epsilon$-optimal solution to the weighted set cover problem.
\end{lem}
\begin{proof}
The proof of this lemma is twofold: (i)~we show that an optimal solution $\bK^\*$ to Problem~\ref{prob:one} gives an optimal solution $\S(\bK^\*)$ to the weighted set cover problem, and (ii)~we show that, if $P(\bK) \leqslant \epsilon\,p^\*$, then $w(\S(\bK)) \leqslant \epsilon\,w(\S^\*)$.

Note that (i)~is proved in Theorem~\ref{th:NP1}. Now for proving~(ii) we use Lemma~\ref{lem:set_prob}. Given
\begin{eqnarray*}
p(\bK) & \leqslant & \epsilon p^\*,\\
w(\S(\bK)) & \leqslant &  \epsilon p^\*,\\
& = &  \epsilon\, w(\S(\bK^\*)),\\
& = &  \epsilon\, w(\S^\*).
\end{eqnarray*}
This completes the proof.
\end{proof}

 Next we prove a negative result that shows that Problem~\ref{prob:one} cannot be approximated up to a constant factor. The inapproximability result holds even for systems whose state bipartite graph $\B(\bA)$ has a perfect matching and $\D(\bA)$ has a single non-top linked SCC. We use the following proposition in the proof.
\begin{prop}\cite[Corollary 3.4]{LunYan:94}\label{prop:inapprox}
For any $0 < b < 1/4$, the set covering problem cannot be approximated within factor $b\,{\rm log\,}N$ in polynomial time unless $NTIME(n^{{\rm poly\,log\,}\,n}) = DTIME(n^{{\rm poly\,log\,}\,n})$, where $N$ denotes the number of items in the universe.
\end{prop}
\noindent{\it Proof~of~Theorem~\ref{th:inapprox}}:
By Proposition~\ref{prop:inapprox} the set cover problem cannot be approximated within factor $b\,{\rm log}\,N$ for $0 < b < 1/4$, where $N$ denotes the cardinality of the universe. Thus the weighted set cover problem  also cannot be approximated within factor $b\,{\rm log}\,N$ since set cover is a special case where all weights are non-zero and uniform. However, by Lemma~\ref{lem:epsilon}, if there exists an approximation algorithm that gives an $\epsilon$-optimal solution to Problem~\ref{prob:one} for a structured system constructed using Algorithm~\ref{alg:set}, then it gives an $\epsilon$-optimal solution to the weighted set cover problem. Thus, since weighted set cover cannot be approximated to factor $b\,{\rm log}\,N$ for $0 < b < 1/4$, Problem~\ref{prob:one} also cannot be approximated even for this special case. Hence, the inapproximability holds for the general structured systems also. This completes the proof. 
\qed

The following result considers a special case, i.e., single input systems with one non-top linked SCC. An example of this instance is single leader multi-agent dynamics. We show that Problem~\ref{prob:one} reduces to a weighted set cover problem in this case. Using this reduction we give an $O(n^2)$ algorithm that gives an $O({\rm log\,}n)$ approximation to Problem~\ref{prob:one}, where $n$ denotes the number of states.
\begin{theorem}\label{lem:single_non_top}
Consider a structured system $\D(\bA, \bB, \bC)$ and a feedback cost matrix $P$. Let $\B(\bA)$ has a perfect matching and $\D(\bA)$ has a single non-top linked SCC. Further, assume that $\bB$ is a single input. Then, there exists an algorithm that gives an $O({\rm log\,} n)$ approximate solution to Problem~\ref{prob:one}, where $n$ denotes the number of states. Moreover, the complexity of the algorithm is $O(n^2)$.
\end{theorem}
\begin{proof}
 Let $\D(\bA)$ has $\beta$ number of non-bottom linked SCCs, $\cN_1^b, \ldots, \cN_\beta^b$. If $\beta=1$, then the graph is irreducible, and the optimal solution to Problem~\ref{prob:one} $\bK^\*$ is given by, $\bK^\*_{1j} = \*$ corresponding to a minimum cost entry in $P$ and 0 otherwise. For $\beta > 1$, we show that any instance of Problem~\ref{prob:one} with single input that has a perfect matching in $\B(\bA)$ and a single non-top linked SCC in $\D(\bA)$ can be reduced to a weighted set cover problem. Let $\beta > 1$ and consider the following weighted set cover problem. Define the universe $\U = \{1,\ldots, \beta \}$, sets $\P = \{\S_1, \ldots, \S_p\}$, where $\S_i = \{k: x_j \in\cN_k^b \mbox{ and } \bC_{ij} = \* \}$. Notice that each set $\S_i$ corresponds to an output $y_i$, i.e, $\S_i$ consists of indices of all non-bottom linked SCCs that has some state vertex which is sensed by $y_i$. The weight function $w$ is defined as $w(i) = P_{1i}$. For a cover $\S$ of the weighted set cover problem, we define the corresponding feedback matrix $\bK(\S)$ as $\bK(\S)_{1j} = \*$ if $\S_j \in \S$. Now, we show that $w(\S) = P(\bK(\S))$ and an optimal solution to the weighted set cover problem gives an optimal solution to Problem~\ref{prob:one}.

Since $\bK(\S)_{1j} = \*$ if $\S_j \in \S$ and $w(i) = P_{1i}$, $w(\S) = P(\bK(\S))$. To prove optimality, we first show that a feasible solution to the above weighted set cover problem gives a feasible solution to Problem~\ref{prob:one}. Let $\S$ be a feasible solution to the weighted set cover problem. Now we need to show that $\bK(\S)$ satisfies conditions~a)~and~b). Since $\B(\bA)$ has a perfect matching, condition~b) is satisfied without using any feedback edges. Thus only condition~a) has to be checked. Since $\D(\bA)$ has a single non-top linked SCC, say $\cN^t$, all other SCCs lie in a path from $\cN^t$ to some non-bottom linked SCC. Also notice that all non-bottom linked SCC must have a feedback edge, otherwise condition~a) is violated. Thus a set of feedback edges that have an edge connecting output connected to some state in the non-bottom linked SCC to the input, for all non-bottom linked SCCs satisfies condition~a). Thus $\bK(\S) \in \K_s$. Now we prove the optimality. Notice that an optimal solution does not include a feedback edge that has a $(y, u)$ edge where $y$ connects from an SCC that is not non-bottom linked. This is because any feasible solution must have feedback edges such that all non-bottom linked SCCs are connected from some outputs in it. Thus any $(y, u)$ edge where $y$ connects from an SCC that is not a non-bottom linked SCC does not add to satisfying condition~a), and hence it is superfluous. Thus such an edge is not present in any optimal solution and hence we need to consider only outputs going out of non-bottom linked SCCs. Let $\S^\*$ be an optimal solution to the weighted set cover problem. We prove optimality of $\bK(\S^\*)$ using a contradiction argument. Suppose $\bK(\S^\*)$ is not an optimal solution to Problem~\ref{prob:one}. Then there exists $\widetilde{\bK}$ such that $\widetilde{\bK} \in \K_s$ and $P(\widetilde{\bK}) < P(\bK(\S^\*))$. Since $\widetilde{\bK} \in \K_s$, condition~a) is satisfied. Let $\widetilde{\bK}_{1j} = \*$ for $j \in \{i_1, \ldots, i_y \}$. The corresponding sets are $\S_{i_1}, \ldots, \S_{i_y}$. Since condition~a) is satisfied, $\{i_1, \ldots, i_y \}$ consists of index of at least one state from all the non-bottom linked SCCs, otherwise there exists some state that does not satisfy condition~a). Thus $\widetilde{\S} = \{\S_{i_1}, \ldots, \S_{i_y}\}$ is a cover. Since $w(\S) = P(\bK(\S))$, we get $w(\widetilde{\S}) < w(\S^\*)$. This contradicts the assumption that $\S^\*$ is an optimal solution to the weighted set cover problem. Thus $\bK(\S^\*)$ is an optimal solution to Problem~\ref{prob:one}. 

Using similar arguments given in the proof of Lemma~\ref{lem:epsilon}, now we can also show that for $\epsilon > 1$, if $w(\S) \leqslant \epsilon\, w(\S^\*)$, then $P(\bK(\S)) \leqslant \epsilon\, P(\bK(\S^\*))$. Thus any approximation algorithm of the weighted set cover problem can be used to solve Problem~\ref{prob:one} on this class of systems. Chavtal gave a linear complexity $O({\rm log\,}N)$ approximation algorithm based on greedy scheme for solving the weighted set cover problem, where $N$ denotes the cardinality of the universe \cite{Chv:79}. Thus by using the same algorithm, we get an $O({\rm log\,}n)$ approximate solution to Problem~\ref{prob:one}, since $\beta = O(n)$. 

Now we give the complexity of the proposed scheme. Finding the SCCs in $\D(\bA)$ has $O(n^2)$ complexity and the greedy algorithm has $O(n)$ complexity. Thus the overall complexity to find an approximate solution is $O(n^2)$. This completes the proof.
\end{proof}
In multi-input case with uniform feedback costs, there exists a similar reduction of Problem~\ref{prob:one} to the set cover problem. Thus there exists $O({\rm log\,}n)$ approximation to Problem~\ref{prob:one} of complexity $O(n^2)$.

In the next section we consider Problem~\ref{prob:one} on two special classes of systems. We show that Problem~\ref{prob:one} can be solved optimally with polynomial complexity for one class of systems using a dynamic programming algorithm. Then we give a polynomial time 2-optimal approximation algorithm for the second class of systems for solving Problem~\ref{prob:one} using the dynamic programming algorithm proposed and a minimum cost perfect matching algorithm.
\section{Line Graph Systems}\label{sec:line}
In this section we consider structured systems whose directed acyclic graph (DAG) obtained by condensing SCCs in $\D(\bA)$ in to nodes is a line graph. In other words, the DAG constructed after condensing SCCs in $\D(\bA)$ to super nodes and connecting these super nodes if there exists an edge connecting two states in those SCCs is a directed path as shown in Figure~\ref{fig:line}.

Let $\{\cC_1, \ldots, \cC_\ell \}$ denote the {\it ordered} set of SCCs in $\D(\bA)$. Note that in this graph there is exactly one non-top linked SCC, $\cC_1$, and exactly one non-bottom linked SCC, $\cC_\ell$. We further assume that $\B(\bA)$ has a perfect matching. Thus condition~b) in Proposition~\ref{prop:SFM} is satisfied and hence solving Problem~\ref{prob:one} optimally is equivalent to satisfying condition~a) optimally. Note that connecting an output $y$ that is connected to $\cC_\ell$ to an input $u$ that is connected to $\cC_1$ may not be optimal to satisfy condition~a) as this connection can be very expensive when compared to the rest of the connections. Further, an optimal solution may consists of connections that cover some of the SCCs multiple times. This can happen if satisfying condition~a) is cheaper that way when compared to others. 

If the feedback costs are uniform, then Problem~\ref{prob:one} is trivial. In that case since $\cC_1$ is the only non-top linked SCC and $\cC_\ell$ is the only non-bottom linked SCC, connecting an output $y$ that connects to $\cC_\ell$ to an input $u$ that connects to $\cC_1$ will satisfy condition~a). Similarly, if the digraph $\D(\bA)$ is irreducible, that is $\D(\bA)$ is a single SCC, then also the solution is trivial. In that case connecting $(y_j, u_i)$ where $P_{ij}$ is the smallest entry in the matrix $P$ is optimal. Thus optimal solution $\bK^\*$ for this case has $\*$ only at one location, i.e., $\bK^\*_{ij}$. Figure~\ref{fig:line} shows a schematic diagram of the line graph whose vertices are SCCs in $\D(\bA)$. We prove that Problem~\ref{prob:one} can be solved in polynomial time for this class of systems. 
\begin{figure}
\begin{center}
\begin{tikzpicture}[->,>=stealth',shorten >=0.5pt,auto,node distance=1.85cm,
                thick,main node/.style={circle,draw,font=\small\bfseries}]

  \node[main node] (1) {$\cC_1$};
  \node[main node] (2) [right of=1] {$\cC_2$};
  \node[main node] (3) [right of=2] {$\cC_{i}$};
  \node[main node] (4) [right of=3] {$\cC_{j}$};
  \node[main node] (5) [right of=4] {$\cC_{\ell}$};

 \draw[] (1) -> (2);
 \draw[dashed] (2) -> (3);
 \draw[dashed] (3) -> (4);
 \draw[dashed] (4) -> (5);    
\end{tikzpicture}
\caption{The DAG of SCCs in $\D(\bA)$ of a structured system}
\label{fig:line}
\end{center}
\end{figure}
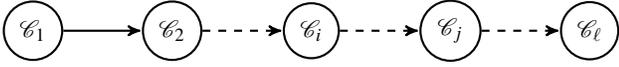

For this section the following assumption holds.
\begin{assume}\label{asm:line}
The DAG of SCCs in $\D(\bA)$ is a line graph.
\end{assume}

We propose a polynomial time algorithm for solving Problem~\ref{prob:one} for structured systems when the bipartite graph $\B(\bA)$ has a perfect matching and Assumption~\ref{asm:line} holds. The proposed algorithm is a dynamic programming algorithm. Since $\B(\bA)$ has a perfect matching, the algorithm aims at achieving condition~a) in Proposition~\ref{prop:SFM} optimally. The pseudo-code of the proposed scheme is presented in Algorithm~\ref{alg:dynamic}. 
\begin{algorithm}[]
  \caption{Pseudo-code of dynamic programming algorithm for solving Problem~\ref{prob:one} on structured systems when $\B(\bA)$ has a perfect matching and Assumption~\ref{asm:line} holds
  \label{alg:dynamic}}
  \begin{algorithmic}
\State \textit {\bf Input:} structured system $(\bA, \bB, \bC)$, feedback cost matrix $P$
\State \textit{\bf Output:} Feedback matrix $\bK^a$
\end{algorithmic}
  \begin{algorithmic}[1]
  \State $\{\cC_1, \ldots, \cC_{\ell}\}$ are the SCCs in $\D(\bA)$
    \State $U_k \leftarrow \{u_i: \bB_{ri} = \* \mbox{~and~} x_r \in \cC_k \}$\label{step:input}
  \State $Y_k \leftarrow \{y_j: \bC_{jr} = \*\mbox{~and~} x_r \in \cC_k \}$\label{step:output}
    \State $\cU_k \leftarrow \cup_{i = 1}^k U_i$\label{step:IPset1}
  \State $\cY_k \leftarrow \cup_{i =k}^\ell Y_i$ \label{step:IPset2}
    \State $W([0]) \leftarrow 0$, $\cS_0 \leftarrow \phi$
\For {$k = 1,\ldots,\ell$}
   \State $W([k])  \leftarrow$ min cost to keep $\{\cC_1, \ldots, \cC_{k}\}$ in cycles \label{step:step_k}
  \State $A_k \leftarrow \{(y_j, u_i):y_j \in \cY_k \mbox{~and~} u_i \in \cU_k \}$ \label{step:edge_k}
\State $t_k(i) \leftarrow {\rm min}_q \{u_i \in \cU_q: (y_j, u_i)\in A_k\}$ \label{step:minU_k}
  \State $W([k]) \leftarrow \mbox{min}_{(y_j, u_i) \in A_{k}}\{P_{ij} + W([t_k(i)-1])  \}$ \label{step:step_k+1} 
   \State If $W([k]) = P_{vw} + W([z])$, then $\cS_k \leftarrow (y_w, u_v) \cup \cS_{z}$ \label{step:edges_selected}
  \EndFor
  \State $\bK^a \leftarrow \{{\bK^a} _{ij} = \*:(y_j, u_i) \in  \cS_\ell \}$ \label{step:feedback}
\end{algorithmic}
\end{algorithm}

Consider a structured system $(\bA, \bB, \bC)$ and cost matrix $P$. We denote the SCCs in $\D(\bA)$ as $\{\cC_1, \ldots, \cC_{\ell}\}$. We define $U_k$ as the set of inputs that connect to some states in $\cC_k$ (see Step~\ref{step:input}). Similarly, we define $Y_k$ as the set of outputs that are connected from some states in $\cC_k$ (see Step~\ref{step:output}). Now we have the following definition.

\begin{defn}\label{def:cover}
An SCC $\cC_k$ is said to be covered if condition~a) is satisfied for all states in $\cC_k$. In other words, an edge $(y_j, u_i)$ covers $\cC_k$ if all the state nodes in $\cC_k$ lie in an SCC with edge $(y_j, u_i)$.
\end{defn}

 Note that connecting some $u_i \in U_k$ to some $y_j \in Y_k$ covers $\cC_k$. However, in addition to these there are other feedback edges that can cover $\cC_k$. To characterize all the feedback edges that cover SCC $\cC_k$, we define sets $\cU_k$ and $\cY_k$. Here $\cU_k$ consists of all inputs that are connected to some states in $\cC_j$'s for $j \leqslant k$. Similarly, $\cY_k$ consists of all outputs that are connected from some states in $\cC_j$'s for $j \geqslant k$. Thus $A_k = \{(y_j, u_i):y_j \in \cY_k \mbox{~and~} u_i \in \cU_k \}$ consists of all edges that cover SCC $\cC_k$ (see Step~\ref{step:edge_k}). The key insight for the dynamic programming based algorithm provided in Algorithm~\ref{alg:dynamic} is given in the following lemma.
\begin{lem}\label{lem:A_k}
Consider a structured system $(\bA, \bB, \bC)$ and cost matrix $P$ given as input to Algorithm~\ref{alg:dynamic}. Let $\bK^\*$ is an optimal solution to Problem~\ref{prob:one} and $\cS^\* = \{(y_j, u_i):\bK^\*_{ij} = \*\}$. Then, for all $k$, $\cS^\* \cap A_k \neq \phi$.
\end{lem}
\begin{proof}
Given $\bK^\*$ is an optimal solution to Problem~\ref{prob:one} and $\cS^\*$ is the corresponding set of minimum cost feedback edges. Thus edges in $\cS^\*$ cover SCCs $\cC_1, \ldots, \cC_\ell$. The sets $\cU_k$ and $\cY_k$ in Algorithm~\ref{alg:dynamic} are constructed in such a way that the set $A_k$, given by $A_k = \{(y_j, u_i):y_j \in \cY_k \mbox{~and~} u_i \in \cU_k \}$, consists of all possible feedback edges that can cover SCC $\cC_k$. Suppose $\cS^\* \cap A_k = \phi$.  Then, the edges in $\cS^\*$ do not cover $\cC_k$. Thus $\bK^\*$ does not satisfy condition~a) in Proposition~\ref{prop:SFM}. Hence, $\bK^\* \notin \K_s$. This contradicts the assumption that $\bK^\*$ is an optimal solution to Problem~\ref{prob:one}. Thus for all $k$, $\cS^\* \cap A_k \neq \phi$.
\end{proof}

 Now for $(y_j, u_i) \in A_k$, $t_k(i)$ is defined as the lowest index $q$ such that $u_i \in \cU_q$ (see Step~\ref{step:minU_k}). Thus $t_k(i) \leqslant k$. The significance of $t_k(i)$ is that edge $(y_j, u_i) \in A_k$ not only covers SCC $\cC_k$, but also cover all the SCCs $\cC_{t_k(i)}, \ldots, \cC_k$. Thus if $(y_j, u_i)$ is present in the set of edges that cover $\cC_1, \ldots, \cC_k$, then the rest of the edges need to cover only $\cC_1, \ldots, \cC_{t_k(i) - 1}$. Now $W([k])$ given in Step~\ref{step:step_k+1} of the algorithm denotes the minimum cost for covering $\cC_1,\ldots, \cC_k$ and $\cS_k$ denotes the corresponding feedback edges (see Step~\ref{step:edges_selected}). The dynamic programming step of the algorithm proceeds as follows.
 
For $k=1$, we start at SCC $\cC_1$. To cover $\cC_1$, we will pick an edge in $A_1$ that is of the least cost. Thus $\cS_1$ consists of a single edge which is from $A_1$. Now we cover $\cC_1$, $\cC_2$ together. Thus an edge in $A_2$ will be present. This edge will connect an output $y_j \in \cY_2$ to an input $u_i$ in $\cU_2$. Suppose $u_i \in U_2$ and $u_i \notin U_1$. Then edge $(y_j, u_i)$ covers only $\cC_2$ and not $\cC_1$. Thus the optimal cost to cover $\cC_1, \cC_2$ is $P_{ij}+ W([1])$. Else if $u_i \in U_2$, then SCCs $\cC_1, \cC_2$ are covered. Then the optimal cost to cover $\cC_1, \cC_2$ is $P_{ij}$. Finally, the minimum cost to cover $\cC_1, \cC_2$ is obtained by finding minimum over all edges in $A_2$. A generic dynamic programming equation is given in Step~\ref{step:step_k+1} of Algorithm~\ref{alg:dynamic}. $\cS_k$ keeps track of the edges required to cover $\cC_1, \ldots, \cC_k$ with the minimum cost. Every stage of the dynamic programming algorithm is updated using Steps~\ref{step:minU_k}~and~\ref{step:step_k+1}. Now the optimal solution to Problem~\ref{prob:one} is obtained using the edges present in $\cS_\ell$ as shown in Step~\ref{step:feedback}. For showing the optimality of Algorithm~\ref{alg:dynamic}, now we formally prove Theorem~\ref{th:twostage}~(i).

\noindent{\it Proof~of~Theorem~\ref{th:twostage}~(i):}
 We prove~(i) using an induction argument. The induction hypothesis is that $W([k])$ is the minimum cost to cover SCCs $\cC_1, \ldots, \cC_k$ and $\cS_k$ is the corresponding optimal set of feedback edges. 

Base step:~We consider $k=1$ as the base step. For $k=1$, $\cU_1 = U_1$. Thus $t_k(i) = 1$. Hence, $W([1]) = \mbox{min}_{(y_j, u_i) \in A_1}\{P_{ij}\}$. Note that here $A_1$ consists of all possible edges that can result in making all state nodes in $\cC_1$ lie in an SCC with a feedback edge. In other words all possible feedback edges that can cover $\cC_1$. Thus the algorithm selects an optimal edge in $A_1$ such that all state nodes in $\cC_1$ lie in an SCC with that feedback edge. Suppose $(y_j, u_i)$ is chosen. Then clearly $u_i \in U_1$ and $y_j \in Y_q$ for some $q \geqslant 1$. Thus condition~a) is satisfied for all states in $\cC_1$ optimally. This proves the base step.

Induction step:~For the induction step we assume that the optimal cost to cover SCCs $\cC_1,\ldots,\cC_{k-1}$ are $W([1]), \ldots, W([k-1])$ respectively. Also, the corresponding edge sets are $\cS_1, \ldots, \cS_{k-1}$ respectively.

 Now we will prove that $W([k])$ is the minimum cost to cover $\cC_1,\ldots, \cC_k$ and $\cS_k$ is the corresponding feedback edge set. Note that $A_k$ consists of all feedback edges that can cover $\cC_k$. Thus an edge in $A_k$ has to be used for covering $\cC_k$. Let $(y_j, u_i) \in A_k$. Note that $(y_j, u_i)$ not only covers $\cC_k$ but also cover $\cC_{t_k(i)}, \ldots, \cC_k$.  Thus the optimal cost to cover $\cC_1, \ldots, \cC_k$ using $(y_j, u_i)$ is $P_{ij} + W([t_k(i) - 1])$. Notice that $W([k])$ is found after performing a minimization over all edges in $A_k$. Since $t_k(i) \leqslant k$ and we assumed that the induction hypothesis is true for $\cC_1, \ldots, \cC_{k-1}$, $W([k])$ is the minimum cost for covering $\cC_1, \ldots, \cC_k$. Further, $\cS_k$ is the union of that edge $(y_j, u_i) \in A_k$ that is selected in the minimization step and $\cS_{t_k(i) - 1}$. Thus $\cS_k$ is the corresponding set of edges of $W([k])$. This completes the proof of~(i).
\qed

Now we consider a class of structured systems where only Assumption~\ref{asm:line} holds, i.e., the bipartite graph $\B(\bA)$ does not have a perfect matching. Since $\B(\bA)$ does not have a perfect matching, condition~b) in Proposition~\ref{prop:SFM} also has to be satisfied using the feedback connections. In this case, we propose a two stage algorithm. The proposed algorithm uses the dynamic programming algorithm explained above and a minimum cost perfect matching algorithm \cite{CorLeiRivSte:01}. The dynamic programming algorithm gives solution $\bK^a$ that satisfies condition~a). The minimum cost perfect matching algorithm gives a solution $\bK^b$ that satisfies condition~b). We prove that combining these together we get a 2-optimal solution to Problem~\ref{prob:one}.

\begin{algorithm}[t]
  \caption{Pseudo-code for solving Problem~\ref{prob:one} on structured systems where Assumption~\ref{asm:line} holds
  \label{alg:twotage}}
  \begin{algorithmic}
\State \textit {\bf Input:} Structured system $(\bA, \bB, \bC)$ and feedback cost matrix $P$
\State \textit{\bf Output:} Feedback matrix $\bK^A$
\end{algorithmic}
  \begin{algorithmic}[1]
    \State Find feedback matrix satisfying condition~a) using Algorithm~\ref{alg:dynamic}, say $\bK^a$ \label{step:soln1}
    \State Construct the bipartite graph $\B(\bA, \bB, \bC, \bK)$
  \State For $e \in \E_X \cup \E_U \cup \E_Y \cup \E_K \cup \E_{\mathbb{U}} \cup \E_{\mathbb{Y}}$ define:
  \State Cost, $ c(e) \leftarrow 
\begin{cases}
P_{ij}, {~\rm for~} e = (u'_i,y_j) \in \E_K,\\
0, ~~   {\rm otherwise}.
\end{cases}$ \label{step:bip_cost}
   \State  Find minimum cost perfect matching of $\B(\bA, \bB, \bC, \bK)$ under cost $c$, say $M^\*$
\State Find feedback matrix satisfying condition~b) optimally using $M^\*$, say $\bK^b$
\State $\bK^b \leftarrow \{{\bK^b}_{ij} = \*: (u'_i, y_j) \in M^\* \}$ 
\label{step:soln2}
\State $\bK^A \leftarrow \{{\bK^A}_{ij} = \* \mbox{ if either }{\bK^a}_{ij} = \* \mbox{ or } {\bK^b}_{ij} \ = \*\}$ \label{step:final}
\end{algorithmic}
\end{algorithm}
The pseudo-code for solving Problem~\ref{prob:one} on a structured system where only Assumption~\ref{asm:line} holds is presented in Algorithm~\ref{alg:twotage}. Firstly, an optimal set of feedback edges that satisfy condition~a) in Proposition~\ref{prop:SFM} is obtained using the dynamic programming algorithm given in Algorithm~\ref{alg:dynamic}. Let $\bK^a$ denote the feedback matrix obtained as solution to the dynamic programming algorithm (see Step~\ref{step:soln1}). Note that this feedback matrix does not guarantee condition~b). To satisfy condition~b) we run a minimum cost perfect matching algorithm on the bipartite graph $\B(\bA, \bB, \bC, \bK)$ with cost function defined as shown in Step~\ref{step:bip_cost}. Let $M^\*$ be an optimal matching obtained and $\bK^b$ is the feedback matrix selected under $M^\*$ (see Step~\ref{step:soln2}). From Proposition~\ref{prop:match}, $\bK^b$ satisfies condition~b) in Proposition~\ref{prop:SFM}. Note that feedback matrix $\bK^A$ obtained by taking element wise union of $\bK^a$ and $\bK^b$ (see Step~\ref{step:final}) satisfies both the conditions in Proposition~\ref{prop:SFM} and hence is a feasible solution to Problem~\ref{prob:one}. We have the following lemma.
\begin{lem}\label{lem:match_cost}
Let $M$ be a perfect matching in $\B(\bA, \bB, \bC, \bK)$ and $\bK(M):= \{\bK(M)_{ij} = \*: (u'_i, y_j) \in M \}$  be the feedback matrix selected under $M$. Then, $c(M) = P(\bK(M))$.
\end{lem}
\begin{proof}
The proof is an immediate consequence of Steps~\ref{step:bip_cost}~and~\ref{step:soln2} in Algorithm~\ref{alg:twotage}.
\end{proof}
 
Now we give the proof of Theorem~\ref{th:twostage}.

\noindent{\it Proof~of~Theorem~\ref{th:twostage}:}
Part~(i) of Theorem~\ref{th:twostage} is given before. Hence we proceed with the proof of part~(ii). For proving~(ii), let $\bK^\*$ be an optimal solution to Problem~\ref{prob:one} with cost $p^\*$. Thus $\bK^\*$ satisfies both the conditions in Proposition~\ref{prop:SFM}. Thus the optimal cost for satisfying each condition individually is at most $p^\*$. Thus
\begin{eqnarray*}
p^\* & \geqslant & P(\bK^a), \\
p^\* & \geqslant & P(\bK^b),  \\
2\,p^\* & \geqslant & P(\bK^a) + P(\bK^b). 
\end{eqnarray*} 
Thus $P(\bK^A) \leqslant 2\,p^\*$.
\qed

The following theorem gives complexities of the two algorithms proposed in this paper for solving Problem~\ref{prob:one}.

\begin{theorem}\label{th:comp1}
Consider a structured system $(\bA, \bB, \bC)$ with $n$ number of states and cost matrix $P$. Then, both Algorithm~\ref{alg:dynamic} and Algorithm~\ref{alg:twotage} has complexity $O(n^3)$.
\end{theorem}

\begin{proof}
Finding the SCCs in $\D(\bA)$ has $O(n^2)$ complexity. Let $m, p$ denote the number of inputs and outputs in the structured system. Then each stage of Algorithm~\ref{alg:dynamic} has to compute at most $mp$ number of values and find the least value amongst them. Note that $m = O(n)$ and $p = O(n)$. Thus each stage of the algorithm is of complexity $O(n^2)$. The maximum number of iterations required is the number of SCCs in $\D(\bA)$ which is at most $n$. Thus complexity of Algorithm~\ref{alg:dynamic} is $O(n^3)$.

Algorithm~\ref{alg:dynamic} has complexity $O(n^3)$ and the minimum cost perfect matching algorithm has complexity $O(n^{2.5})$. Combining both, Algorithm~\ref{alg:twotage} has complexity $O(n^3)$ and this completes the proof.
\end{proof}

\begin{rem}\label{rem:spanning_tree}
In the DAG of SCCs of $\D(\bA)$, if there exists a spanning tree that is a line graph, then all the analysis and results discussed in this paper still holds. Figure~\ref{fig:spanning} shows a schematic diagram of DAG of SCCs of $\D(\bA)$ of such a system. In such a case, one needs to look at only that particular spanning tree for solving Problem~\ref{prob:one}. This gives a generalization of the structured systems that are studied in this paper.  
\end{rem}

\begin{figure}
\begin{center}
\begin{tikzpicture}[->,>=stealth',shorten >=0.5pt,auto,node distance=1.5cm,
                thick,main node/.style={circle,draw,font=\small\bfseries}]

  \node[main node] (1) {$\cC_1$};
  \node[main node] (2) [right of=1] {$\cC_2$};
  \node[main node] (3) [right of=2] {$\cC_{3}$};
  \node[main node] (4) [right of=3] {$\cC_{4}$};
  \node[main node] (5) [right of=4] {$\cC_{5}$};
  \node[main node] (6) [right of=5] {$\cC_{6}$};

 \draw[] (1) -> (2);
 \draw[] (2) -> (3);
 \draw[] (3) -> (4);
 \draw[] (4) -> (5);
 \draw[] (5) -> (6); 
\path[every node/.style={font=\sffamily\small}]
(1) edge[bend left = 40] node [left] {} (6)
(2) edge[bend left = 40] node [left] {} (4) 
(3) edge[bend left = 40] node [left] {} (5)
(2) edge[bend right = 40] node [left] {} (5)
(4) edge[bend right = 40] node [left] {} (6)
(1) edge[bend right = 40] node [left] {} (4);  
 \end{tikzpicture}
\caption{The line graph DAG corresponding to $\D(\bA)$}
\label{fig:spanning}
\end{center}
\end{figure}
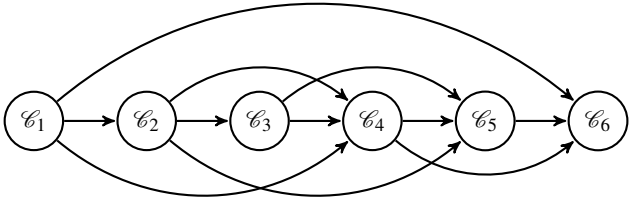
In the next section we explain the dynamic programming algorithm proposed in the paper using an illustrative example.
\section{Illustrative Example}\label{sec:eg}
In this section we describe the proposed dynamic algorithm using an example. Figure~\ref{fig:eg1} denote the digraph of a structured system whose cost matrix is given by
\begin{equation*}
P = 
\begin{bmatrix}
2 & 10 & 100 \\
7 & 8 & 5 \\
9 & 5 & 50  \\
10 & 11 & 13 \\
\end{bmatrix}.
\end{equation*}

There are four SCCs: $\cC_1 = \{x_1, x_2, x_3\}$, $\cC_2 = \{x_4, x_5\}$, $\cC_3 = \{x_6\}$ and $\cC_4 = \{x_7, x_8, x_9, x_{10}\}$. Also, $\B(\bA)$ has a perfect matching. Here $U_1 = \{u_1, u_2\}$, $U_2 = \{u_2 \}$, $U_3 = \{u_3\}$,  $U_4 = \{u_3, u_4\}$. Similarly, $Y_1 = \{y_1\}$, $Y_2 = \phi$, $Y_3 = \{y_2\}$ and $Y_4 = \{y_3\}$. Subsequently, $\cU_1 = \{u_1, u_2\}$, $\cU_2 = \{u_1, u_2\}$, $\cU_3 = \{u_1, u_2, u_3\}$ and $\cU_4 = \{u_1, u_2, u_3, u_4\}$, $\cY_1 = \{y_1, y_2, y_3 \}$, $\cY_2 = \{ y_2, y_3 \}$, $\cY_3 = \{ y_2, y_3 \}$ and $\cY_4 = \{y_3 \}$. 
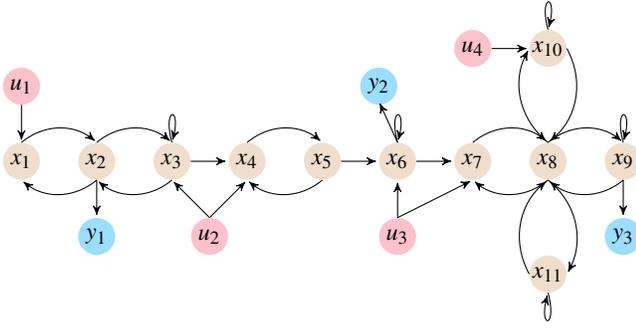
\begin{figure}
\begin{center}
\begin{tikzpicture}[->,>=stealth',shorten >=1pt,auto,node distance=1.85cm, main node/.style={circle,draw,font=\scriptsize\bfseries}]
\definecolor{myblue}{RGB}{80,80,160}
\definecolor{almond}{rgb}{0.94, 0.87, 0.8}
\definecolor{bubblegum}{rgb}{0.99, 0.76, 0.8}
\definecolor{columbiablue}{rgb}{0.61, 0.87, 1.0}

  \fill[bubblegum] (-1,1) circle (7.0 pt);
  \fill[bubblegum] (1.5,-1) circle (7.0 pt);
  \fill[bubblegum] (4,-1) circle (7.0 pt);
  \fill[bubblegum] (5,1.5) circle (7.0 pt);
  \fill[columbiablue] (7,-1.0) circle (7.0 pt);
  \fill[columbiablue] (0,-1.0) circle (7.0 pt);
  \fill[columbiablue] (3.75,1.0) circle (7.0 pt);
  
  \fill[almond] (-1,0) circle (7.0 pt);
  \fill[almond] (0,0) circle (7.0 pt);
  \fill[almond] (1.0,0) circle (7.0 pt);
  \fill[almond] (2.0,0) circle (7.0 pt);
  \fill[almond] (3.0,0) circle (7.0 pt);
  \fill[almond] (4.0,0) circle (7.0 pt);
  \fill[almond] (5.0,0) circle (7.0 pt);
  \fill[almond] (6.0,0) circle (7.0 pt);
  \fill[almond] (7.0,0) circle (7.0 pt);
  \fill[almond] (6.0,1.5) circle (7.0 pt);
  \fill[almond] (6.0,-1.5) circle (7.0 pt);

   \node at (-1,1) {\small $u_1$};
   \node at (1.5,-1) {\small $u_2$};
   \node at (4,-1) {\small $u_3$};
   \node at (5,1.5) {\small $u_4$};
   
   \node at (0,-1.0) {\small $y_1$};
   \node at (3.75,1.0) {\small $y_2$};
   \node at (7,-1.0) {\small $y_3$};   
      
  \node at (-1,0) {\small $x_1$};
  \node at (0,0) {\small $x_2$};
  \node at (1,0) {\small $x_3$};
  \node at (2,0) {\small $x_4$};
  \node at (3,0) {\small $x_5$};
  \node at (4,0) {\small $x_6$};
  \node at (5,0) {\small $x_7$};
  \node at (6.0,0) {\small $x_8$};
  \node at (7,0) {\small $x_9$};
  \node at (6,1.5) {\small $x_{10}$};
  \node at (6,-1.5) {\small $x_{11}$};

  \draw (-1,0.75)  ->   (-1,0.25);
  \draw (1.5,-0.75)  ->   (1,-0.25);
  \draw (1.5,-0.75)  ->   (2,-0.25);
  \draw (4,-0.75)  ->   (4,-0.25);
  \draw (4,-0.75)  ->   (5,-0.25);
  \draw (5.25,1.5)  ->   (5.75,1.5);
  
  \draw (1.25,0)  ->   (1.75,0);
  \draw (3.25,0)  ->   (3.75,0);
  \draw (4.25,0)  ->   (4.75,0);
  
  \draw (0,-0.25) ->   (0,-0.75);
  \draw (7,-0.25) ->   (7,-0.75);
  \draw (4,0.25) ->   (3.75,0.8);

\path[every node/.style={font=\sffamily\small}]
(-1,0.25) edge[bend left = 40] node [left] {} (0,0.25)
(0,-0.25) edge[bend left = 40] node [left] {} (-1,-0.25)
(0,0.25) edge[bend left = 40] node [left] {} (1,0.25)
(1,-0.25) edge[bend left = 40] node [left] {} (0,-0.25)
(2,0.25) edge[bend left = 40] node [left] {} (3,0.25)
(3,-0.25) edge[bend left = 40] node [left] {} (2,-0.25)
(1,0.25) edge[loop above] (1.5,0.25)
(4.0,0.25)edge[loop above] (4.75,0.25)
(7,0.25) edge[loop above]  (7.5,0.25)
(6,1.75) edge[loop above]  (6.5,1.75)
(6,-1.75) edge[loop below] (6.5,-1.25)
(5,0.25) edge[bend left = 40] node [left] {} (6,0.25)
(6,-0.25) edge[bend left = 40] node [left] {} (5,-0.25)
(6,0.25) edge[bend left = 40] node [left] {} (7,0.25)
(7,-0.25) edge[bend left = 40] node [left] {} (6,-0.25)
(6,0.25) edge[bend left = 40] node [left] {} (5.75,1.5)
(6.25,1.5) edge[bend left = 40] node [left] {} (6,0.25)

(6,-0.25) edge[bend left = 40] node [left] {} (6.25,-1.5)
(5.75,-1.5) edge[bend left = 40] node [left] {} (6,-0.25);
\end{tikzpicture}
\caption{Digraph $\D(\bA, \bB, \bC)$ of the structured system considered in the example illustrated in Section~\ref{sec:eg}}
\label{fig:eg1}
\end{center}
\end{figure} 

In stage~1, $W([1]) = {\rm min}\{2,10,100,7,8,5 \} = 2$ and $\cS_1 = (y_1,u_1)$. In stage~2, $W([2]) = {\rm min}\{10+0, 100+0, 8+0, 5+0 \} = 5$ and $\cS_2 = (y_3,u_2)$. In stage~3, $W([3]) = {\rm min}\{10+0, 100+0, 8+0, 5+0, 5+5, 50+5 \} = 5$ and $\cS_3 = (y_2,u_1)$.  In stage~4, $W([4]) = {\rm min}\{100+0, 5+0, 50+5, 13+5 \} = 5$ and $\cS_4 = (y_3,u_2)$. Thus connecting $(y_3, u_2)$ is an optimal connection in this example. Thus $F^\*_a = 
\begin{bmatrix}
0 & 0 & 0 \\
0 & 0 & \* \\
0 & 0 & 0  \\
0 & 0 & 0 \\
\end{bmatrix}$ is an optimal solution to Problem~\ref{prob:one}.
\section{Conclusion and Future Work}\label{sec:conclu}
This paper deals with feedback selection of structured systems for arbitrary pole placement when each feedback edge is associated with a cost. Our aim is to optimally select minimum cost feedback matrix such that arbitrary pole placement is possible. This problem cannot be solved in polynomial time unless P=NP. In this paper we give a reduction of a well studied NP-hard problem, the weighted set cover problem to an instance of Problem~\ref{prob:one}. We also show that Problem~\ref{prob:one} cannot be approximated to factor $b\,{\rm log\,}n$, where $n$ denotes the number of states in the system and $0 < b < \frac{1}{4}$. Due to the NP-hardness of the problem we considered a special class of systems, where the directed acyclic graph of SCCs of $\D(\bA)$ is a line graph and $\B(\bA)$ has a perfect matching. We gave a polynomial time optimal algorithm based on dynamic programming for solving Problem~\ref{prob:one} on this class of systems. Further, we studied another special class of systems after relaxing the perfect matching assumption, and gave a 2-optimal polynomial time algorithm for solving Problem~\ref{prob:one} on these class of systems. Finding a good approximation algorithm for a general system still remains an open problem and is a part of future research.
\bibliographystyle{myIEEEtran}  
\bibliography{Line_Graph} 
\end{document}